\documentclass[graybox]{svmult}

\newtheorem{ques}{Question}

\usepackage{amssymb,amsmath}

\usepackage{mathptmx}       
\usepackage{helvet}         
\usepackage{courier}        
\usepackage{type1cm}        
\usepackage{makeidx}         
\usepackage{graphicx}        
\usepackage{multicol}        
\usepackage[bottom]{footmisc}

\newcommand{\Mod}[1]{\ (\mathrm{mod}\ #1)}

\makeindex             

\begin{document}

\title*{On distinct consecutive differences}
\author{Imre Ruzsa, George Shakan, J\'ozsef Solymosi, and Endre Szemer\'edi}
\institute{Imre Ruzsa \at Alfr\'ed R\'enyi Institute of Mathematics, Hungarian Academy of Sciences, \email{ruzsa@renyi.hu}
\and George Shakan \at Department of Mathematics at University of Oxford,
\email{shakan@maths.ox.ac.uk}
\and Jozsef Solymosi \at Department of Mathematics,
University of  British Columbia,   \email{solymosi@math.ubc.ca}
\and Endre Szemer\'edi \at Alfr\'ed R\'enyi Institute of Mathematics, Hungarian Academy of Sciences, \email{szemered@renyi.hu} }

%
%
\maketitle

\abstract{We show that if $A=\{a_1 < a_2 < \ldots < a_k\}$ is a set of real numbers such that the differences of the consecutive
elements are distinct, then for and finite $B \subset \mathbb{R}$,
$$|A+B|\gg |A||B|^{1/2}.$$ The bound is tight up to the constant.}

\section{Introduction}
\label{sec:1}
Given two sets $A, B \subset \mathbb{R}$, the {\em sumset}
of $A$ and $B$ is 
\[
A+B = \{a+b: a \in A \text{ and } b \in B\}.
\]
We say a finite set $A=\{a_1 < a_2< \cdots <a_k\}$ of real numbers
with the property that
\begin{equation}
a_i-a_{i-1}<a_{i+1}-a_i,
\end{equation}
for any $1<i<k$ is {\it convex}. There is the following conjecture of Erd\H{o}s, which motivates the current work. We use Vinogradov's notation so that $a \ll b$ means $a = O(b)$.

\begin{conjecture}\label{conv} Let $A \subset \mathbb{R}$ be convex. Then for any $\epsilon > 0$, $$|A+A| \gg_{\epsilon} |A|^{2 - \epsilon}.$$
\end{conjecture}

Conjecture~\ref{conv} asserts that the local hypothesis of being convex implies the global consequence of having a large sumset. The following example which shows that some form of the $\epsilon$ is necessary.

\begin{example}\label{squares} Let $k$ be a positive integer and $A =\{i^2 : 1 \leq i \leq k\}$. Then $A+A$ is contained in the set of integers of size $\leq 2k^2$ that can be represented as the sum of two squares. Fermat showed that such integers must have a prime factorization where all the primes equivalent to 3 modulo 4 appear to an even power. The sieve implies $$|A+A| \ll  k^2\prod_{\substack{ p \equiv 1 \Mod 4 \\ p \leq k}} (1 - 1/p) \ll k^2 / \log (k)^{1/2}.$$
\end{example}

One trivial obstruction to a sumset being small is that it is a large subset of an arithmetic progression. It is easy to see that any convex subset of an arithmetic progression has size $\ll \sqrt{n}$, which supports Conjecture~\ref{conv}. On the other hand, no such argument can establish the growth demanded by Conjecture~\ref{conv}.

 The first progress towards Conjecture~\ref{conv} is due to Hegyv\'ari
\cite{HE}, who proved that if $A$ is convex then 
$$|A+A|\gg k\frac{\log
k}{\log\log k}.$$ Hegyv\'ari's result was later improved by Elekes, Nathanson, and
Ruzsa \cite{ENR}, who showed if $A$ is convex then
\begin{equation}\label{ENR} |A+B|\gg k^{3/2}, \end{equation} for any set $B$ with $|B|=k.$ Garaev \cite{Ga} later provided a different proof in the case $B =A$.  Solymosi and Szemer\'edi \footnote{unpublished result} proved that there is a constant $c>0$ such that if $A$ is a large enough convex set of numbers then
\[
|A+A|\gg |A|^{3/2+c}.
\]

Schoen and Shkredov improved the result in \cite{SchSh} by showing that the constant $c$ in the above inequality can be arbitrarily close to $1/18$ (and $1/10$ if sumset is replaced by difference set). The current best result towards Conjecture~\ref{conv} is that $c$ can be taken arbitrarily close to $5/74$, which follows from the Schoen-Shkredov argument and a later paper of Shkredov \cite[Theorem 2]{Sh}.

We extend this result of Elekes, Nathanson and Ruzsa, \eqref{ENR}, to sets with distinct consecutive
differences. We say a set $A$ has distinct consecutive differences if for
any $1\leq i,j\leq k,$ $a_{i+1}-a_i=a_{j+1}-a_j$ implies $i=j.$

\begin{theorem}\label{main}
Let $A$ and $B$ be finite sets of real numbers. If $A$ has distinct consecutive differences,
then
\[
|A+B| \gg |A| |B|^{1/2}
\]
In particular, if $|A| = |B|$ then
\[
|A+B| \gg |A|^{3/2}
\]
\end{theorem}

The basic idea behind the proof is the following. The sumset $A+B$
consists of $|B|$ translates of $A$. The translates of two
consecutive elements of $A$ are typically not ``far" from each
other in the sumset $A+B$. Also, from a translate of two
consecutive elements, $b+a_i$,$b+a_{i+1}$ we can recover the value
of $b$, since all of the consecutive differences are distinct.
Then the number of ``close" pairs in $A+B$ should be large, around
$|A||B|$, therefore $A+B$ is also large.

In the second part of the paper we extend Theorem~\ref{main} for two sets.
As an application we show that for any convex function $F$, and
finite sets of real numbers, $A,B,$ and $C$, if $|A|=|B|=|C|=n,$
then $\max\{|A+B|,|F(A)+C|\geq cn^{5/4}.$

Along the same lines of the proof, one can prove a statistical
analog of Theorem~\ref{main} which we state without working out the details of the proof.

\begin{theorem}
 Let $A=\{a_1 < a_2, \cdots <a_k\} \subset \mathbb{R}$. Suppose the set of consecutive differences
$$D=\{a_{i+1}-a_i : 1\leq i \leq k-1\},$$ is large, that is $|D|\geq \delta
|A|$. Then for any finite $B \subset \mathbb{R}$,
$$|A+B|\gg_{\delta} |A||B|^{1/2}.$$
\end{theorem}

Parts of this work was available earlier in unpublished manuscripts, so it received some references in further works, such as  \cite{SchSh,LiRo,LiSha}.

\section{Distinct consecutive differences}
\label{sec:2}

\subsubsection{Proof of Theorem \ref{main}}

\begin{proof} Let $$A = \{ a_1 < \cdots < a_k\},$$ and 
$$S = A+B = \{c_1 < \cdots < c_r\}.$$
Thus $r=|A+B|$. Fix $b \in B$ and $1 \leq i \leq k-1$ and set 
$$J_b(i) : = (a_i +b , a_{i+1} + b] \cap S .$$
Thus the $J_b(i)$ are disjoint as $i$ varies and 
\begin{equation}\label{first} \sum_{i=1}^{k-1} |J_b(i)| \leq |A+B|\end{equation}
We now say $J_b(i)$ is {\it good} if $|J_b(i)| \leq 100 |A+B| k^{-1}$. By \eqref{first} and pigeon-hole, we see that for a fixed $b$, the number of good $J_b(i)$ is $\gg k$. Thus the total number of good $J_b(i)$ is $\gg k |B|$. 
Furthermore the sets $J_{b}(i)$ are distinct. Indeed, given the interval $(a_i + b , a_{i+1} +b]$, we can recover $a_{i+1} - a_i$ by subtracting the two end points. Since $A$ has distinct consecutive differences, this allows as to recover $a_i , a_{i+1}$ and then $b$. 

On the other hand, the number of choices of $\leq 100 |A+B| k^{-1}$ consecutive elements in $S$ is 
$$\ll |A+B|^2 k^{-1},$$ which is also an upper bound for the number of good $J_b(i)$. Combining our upper and lower bounds for the number of good $J_b(i)$, we find
$$k|B| \ll |A+B|^2 k^{-1},$$ as desired. \qed
\end{proof}

\subsubsection{Distinct pairs of consecutive differences}
For an application, it is useful to extend Theorem~\ref{main} to a more general setting. Let 
$$A = \{a_1 < a_2< \cdots <  a_k\},$$ and 
$$A' = \{a'_1 < a'_2< \cdots < a'_k\},$$ be nonempty sets of real numbers. For $1 \leq i \leq k-1$, let 
$$d_i = a_{i+1}-a_i , \ \ \  d'_i = a'_{i+1}-a'_i.$$
The sets $A$ and $A'$ have {\em distinct pairs of consecutive differences} if the ordered pairs $\left(d_i,d'_{i}\right)$ are distinct.

\begin{theorem}\label{pairs}
Let $A$ and $A'$ be finite sets of real numbers with $k$ elements and distinct pairs of consecutive differences.  Let $B$, and $B'$ be arbitrary finite sets of real numbers. Then 
\[
|A+B|\cdot|A'+B'| \gg \left( k^3|B||B'|\right)^{1/2}.
\]
If $k = |A| = |A'| = |B| = |B'|,$ then
\[
|A+B|\cdot|A'+B'| \gg k^{5/2}.
\]
\end{theorem}

The proof is a two dimensional version of the proof of Theorem~\ref{main}, which we now give. 

\begin{proof}
Let 
$$S = A+B  = \{c_1 < \cdots < c_r\},$$
and 
$$S' = A' + B' = \{c'_1 < \cdots < c'_{r'}\}.$$
Fix $b \in B$ and $b' \in B'$ and $1 \leq i \leq k-1$ and set $$J_{b,b'}(i) = (a_i +b , a_{i+1}+b] \cap S \times (a'_{i}+b'  , a'_{i+1}+b'] \cap S' \subset S \times S'.$$
We let $\pi_1 : \mathbb{R}^2 \to \mathbb{R}$ be projection onto the first coordinate and $\pi_2$ be projection to the second. We say $J_{b,b'}(i)$ is {\em good} if 
$$|\pi_1( J_{b,b'}(i))| \leq 100 |A+B| k^{-1}, \ \ \ |\pi_2( J_{b,b'}(i))| \leq 100 |A'+B'| k^{-1}.$$
For a fixed $b \in B$ and $b' \in B'$,
$$\sum_{i=1}^{k-1} |\pi_1( J_{b,b'}(i))| \leq |A+B| ,  \ \ \ \sum_{i=1}^{k-1} |\pi_2( J_{b,b'}(i))| \leq |A'+B'|.$$
Thus by pigeon-hole, the number of good $J_{b,b'}(i)$ is $\gg k$. Allowing $b$ and $b'$ to vary, we conclude the total number of good $J_{b,b'}(i)$ is $\gg k |B| |B'|$. Furthermore, the sets $J_{b,b'}(i)$ are distinct. Indeed given intervals $(a_i +b , a_{i+1}+b]$ and $(a'_{i}+b'  , a'_{i+1}+b'] $, we can recover $a_{i+1} - a_i$ and $a'_{i+1} - a'_{i}$. Since $A$ and $A'$ have distinct pairs of consecutive differences, we may recover $a_i , a_{i+1} , a'_i , a'_{i+1}$ and then $b,b'$. 

On the other hand, the number boxes $I \times I'$ in  $S \times S'$ satisfying 
$$|I \leq 100 |S|k^{-1} , \ \ \ |I'| \leq 100 |S'|k^{-1},$$ is $\ll |S|^2 |S'|^2 k^{-2}$. Thus the number of good $J_{b,b'}(i)$ is $\ll |S|^2 |S'|^2 k^{-2}$. Combining our lower and upper bounds for the number of good $J_{b,b'}(i)$ we find $$k |B||B'| \ll |A+B|^2 |A'+B'|^2 k^{-2},   $$
which completes the proof. \qed
\end{proof} 

A simple consequence of Theorem \ref{pairs} is the following result, which
was first proved by Elekes, Nathanson, and Ruzsa \cite{ENR}.

\begin{theorem}\label{convex}
 For any strictly convex real function $F$, and finite sets of
real numbers, $A,B,$ and $C$, if $|A|=|B|=|C|=k,$ then
$$\max\{|A+B|,|F(A)+C|\gg k^{5/4}.$$ In particular, 
$$|A+F(A)|\gg k^{5/4}.$$
\end{theorem}

\begin{proof}
The two sets, $A$ and $F(A),$ have distinct pairs of consecutive
differences. For the second inequality set $B=F(A)$ and $C=A.$ \qed
\end{proof}

Based on the work of Schoen and Shkredov \cite{SchSh},  Li and Roche-Newton \cite{LiRo} improved the bound in Theorem \ref{convex}. 

\section{A construction for the lower bound}
In this section we show that the bound in Theorem \ref{main} is tight up to a constant multiplier.

 Let $S$ be a Sidon set, that is a set for which all the nonzero differences are distinct. Suppose further that $|S|$ is odd and let us choose the elements of $S$ to be positive and also satisfying
 $$\max_{s\in S}{(s)}< 1/2,$$ 
 for all $s\in S.$ Then there is a list $L$ of the elements of
$S$ with repetitions consisting of $k= 2{|S| \choose 2}$ elements,
such that the consecutive elements have distinct differences.
($L=(s_1,s_2,\ldots , s_k)$ where $s_{i+1}-s_i=s_{j+1}-s_j$
implies that $i=j.$) Indeed, we may follow a directed eulerian circuit in the 
complete graph $K_{|S|}$ where the vertices are labeled by the elements of $S$.

Now we are ready to define $A$ which is the sumset of $S$ and $[k]$:
$$A=\{i+s_i : 1\leq i \leq k\}$$

The set $A$ has the property that the consecutive differences are
distinct, as they are of the form $1+(s_{i+1}-s_i).$ Let us set $B=[k]$ so that $|A| = |B|$. Then 
$$A+B\subset [2k]+S,$$ and so
\[
|A+B|\leq 2k|S|\ll |A|^{3/2}.
\]

Note in the above example $B$ has a much different structure than $A$. This motivates the following question.

\begin{ques} How small can $|A+A|$ be for sets $A$ of size $k$ with distinct consecutive differences?
\end{ques}

\section{Convex sets and $|A+A -A|$}\label{3}

In this section we provide a simple argument that shows a convex set cannot have additive structure.

\begin{proposition}\label{conv1} 
Suppose $A$ is convex and $A' \subset A$. Then 
\begin{equation}\label{33} |A' + A -A| \gg |A'| |A|.\end{equation}
In particular
$$|A-A| |A+A| \gg |A|^{3}.$$
\end{proposition} 
Note that \eqref{33} is best possible, as is seen from Example~\ref{squares}. 
\begin{proof}
We let 
$$A = \{a_1 < \cdots < a_k\}.$$
We prove the first statement in the case $A' = A$ and the general case follows similarly. Let $1 \leq j \leq k-1$. Then the $j$ elements
$$a_j + a_2 - a_1 < \cdots <  a_j + a_{j+1} - a_j,$$ all lie in the interval $(a_j , a_{j+1}]$. Thus $$|A+A -A| \geq \sum_{j=1}^{k-1} j \gg k^2.$$ 
For the second statement, by \cite[Corollary 1.5]{KS}, there is a set $A' \subset A$ such that $|A'| \geq |A|/2$ and  $$|A' + A -A| \ll \left( \frac{|A+A| |A-A|}{|A|} \right) ,$$ and the result follows now from \eqref{33}. \qed
\end{proof}

While the argument is simple, it is not robust. For instance, we cannot prove a statistical analog of \eqref{33}. The proof of Proposition~\ref{conv1} can be modified to handle the case where $A$ is only assumed to have distinct consecutive differences, already hinting at Theorem~\ref{main}.

\section{Difference Sets of Convex Sets} 

In this section we prove Schoen and Shkredov's \cite{SchSh} bound for difference sets of convex sets, slightly modifying some details. We choose to work with difference sets, as there are additional technicalities for sumsets. We say $b \gtrsim a$ if $a = O(b \log^c |A|)$ for some $c >0$. 

\begin{theorem}[\cite{SchSh}]\label{SS} Let $A$ be a convex set. Then 
$$|A-A| \gtrsim |A|^{8/5} $$
\end{theorem} 

Before beginning the proof, we recall the $k^{\rm th}$ order energy of sets $A$ and $B$ is defined as 
$$E_k(A,B) : = \sum_{x} r_{A-B}(x)^k,$$ where 
$$r_{A-B} (x) = \#\{(a,b) \in A \times B : x = a-b\}.$$ 
We set $E_k(A) : = E_k(A,A)$. Using Szemer\'edi-Trotter, it was shown in \cite{SchSh}, building upon the main idea of \cite{ENR}, that if $A$ is convex then
\begin{equation}\label{E3} E_3(A,B) \lesssim |A| |B|^2.\end{equation}
Note that \eqref{E3} is not true if we merely assume that $A$ has distinct consecutive differences as the following example demonstrates.

\begin{example}
Let $k$ be a positive integer (divisible by 10) and 
$$A = d \cdot [k/10] \cup d' \cdot [k/10].$$ 
For appropriately chosen $d$ and $d'$ (i.e. $d = k$ and $d' = k+1$), we have that $A$ has distinct consecutive differences. On the other hand $$E_3(A) \geq E_3([k]) \gg |A|^4.$$
\end{example}

\begin{proof}[Theorem~\ref{SS}] We set $K = |A-A| |A|^{-1}$. By \eqref{E3}, we have 
\begin{equation}\label{E32} E_3(A) \lesssim |A|^3.\end{equation} 
On the other hand, $E_3(A)$ is the number of solutions to 
\begin{equation}\label{eq} a-b = c-d = e-f , \ \ \ a,b,c,d,e,f \in A. \end{equation}
We let 
$$\Delta(A) = \{(a,a) : a \in A\} \subset A \times A.$$ 
Then \eqref{eq} implies
\begin{equation}\label{same} E_2(\Delta(A) , A \times A) = E_3(A).\end{equation}
By a dyadic decomposition there is a $\Delta \geq 1$ such that 
$$|P| \Delta \gtrsim |A|^2, \ \ \ P : = \{x : \Delta \leq r_{A-A}(x) \leq 2 \Delta\}.$$
We define a graph $G = G(\Delta)$ on $A \times A$ such that the edges are 
$$G = \{(a,b) \in A^2 : a - b \in P\}.$$
Then it follows that 
\begin{equation}\label{delta}|G| \gtrsim |A|^2 \ \ \ \Delta \gtrsim \frac{|A|}{K}. \end{equation}
By Cauchy-Schwarz,
$$|A|^6 \lesssim \left(\sum_{x,y} r_{\Delta(A)  -A \times_G A}(x,y) \right)^2 \leq E_2(\Delta(A) , A \times A) |\Delta(A) -(A \times_G A)|.$$
Thus by \eqref{same} and \eqref{E32}, we find that 
\begin{equation}\label{low} |\Delta(A) - A \times_G A| \gtrsim |A|^3.\end{equation}
Now given $(u,v) \in \Delta(A) - A \times_G A$, we write 
$$(u,v) = (c,c) - (a,b), \ \ \ a,b,c \in A, \ \ (a,b) \in G$$
and it follows that 
$$u-v = a-b.$$
But by the definition of $G$, 
$$r_{A-A}(u-v) = r_{A-A}(a-b) \geq \Delta.$$
Combining with \eqref{delta}, we conclude for each $(u,v) \in \Delta(A) - A \times_G A$,
$$r_{A-A}(u-v) \gtrsim |A|K^{-1},$$
 and so by \eqref{low}, 
\begin{equation}\label{AD} |A|^4 K^{-1} \lesssim \sum_{u,v \in \Delta(A) - A \times_G A} r_{A-A} (u-v) .\end{equation}
We set $D = A-A$ and since $u,v \in D$, we find 
 $$\frac{|A|^4}{K}  \lesssim E_2(A,D) \leq E_3(A,D)^{1/2} |A|^{1/2} |D|^{1/2}.$$
Applying \eqref{E3} to the right hand size, we conclude
\begin{equation}\label{last} \frac{|A|^4}{K}  \lesssim |A|^{5/2} K^{3/2}.\end{equation}
Theorem~\ref{SS} now follows from simplification. 
\qed 
\end{proof}

It is only in \eqref{last} of Theorem~\ref{SS} that we utilize \eqref{E3} for a set $B \neq \pm A$.

\begin{acknowledgement}
Research was supported by the OTKA K 119528 grant. The work of the third author has received funding from the European Research Council (ERC) under the European Unions Horizon 2020 research and innovation programme (grant agreement No 741420, 617747, 648017). His research is also
supported an NSERC grant. The second author is supported by Ben Green's Simons Investigator Grant 376201.The authors thank Mel Nathanson for his help on writing this paper.
\end{acknowledgement}
%

%
%
%

\end{document}